\documentclass[12pt,reqno]{amsart}%
\usepackage{amsmath, amsfonts, amssymb, amsthm, amscd, amsbsy}
\usepackage{fancyhdr}
\usepackage[usenames,dvipsnames,svgnames,x11names,hyperref]{xcolor}
\usepackage{geometry}
\usepackage{graphicx}
\usepackage[pagebackref]{hyperref}%
\usepackage{amsmath}%
\usepackage{amsfonts}%
\usepackage{amssymb}
\usepackage{enumerate}

\usepackage[cmtip,all]{xy}
\usepackage{graphicx}
\usepackage{amssymb}
\usepackage{mathrsfs}
\usepackage{amsfonts}
\usepackage{amsmath}

\providecommand{\U}[1]{\protect\rule{.1in}{.1in}}
\hypersetup{backref=true,
pagebackref=true,
hyperindex=true,
colorlinks=true,
breaklinks=true,
urlcolor=NavyBlue,
linkcolor=Fuchsia,
bookmarks=true,
bookmarksopen=false,
filecolor=black,
citecolor=ForestGreen,
linkbordercolor=red
}
\newtheorem{theorem}{Theorem}[section]
\newtheorem*{theorem*}{Theorem}
\theoremstyle{plain}

\newtheorem{lemma}{Lemma}[section]

\newtheorem{remark}{Remark}[section]

\numberwithin{equation}{section}

\allowdisplaybreaks

\def\dim{\operatorname{dim}}

\begin{document}

\title[The optimal constant in the $L^{2}$ Folland-Stein inequality]{The optimal constant in the $L^{2}$ Folland-Stein inequality on the H-type group}


\author{Qiaohua  Yang}
\address{School of Mathematics and Statistics, Wuhan University, Wuhan, 430072, People's Republic of China}

\email{qhyang.math@whu.edu.cn}

\thanks{The work was partially supported by the National Natural
Science Foundation of China(No.11201346).}


\subjclass[2000]{Primary:  43A80; 46E35; 22E25.}



\keywords{Folland-Stein inequality; Heisenberg group; H-type group; best constant}

\begin{abstract}
We  determine the optimal constant in the $L^{2}$ Folland-Stein inequality on the H-type group, which
  partially confirms the conjecture
given by  Garofalo and  Vassilev (Duke Math. J., 2001). The proof is inspired by the work of Frank and Lieb (Ann.
of Math., 2012)
and Hang and Wang.
  \end{abstract}

\maketitle


\section{Introduction}
Let $G$ be a stratified,  simply connected nilpotent Lie group (in short  a Carnot group) of
step $r$. Denote by $\mathfrak{g}$ the  Lie algebra of $G$. It is known that
$\mathfrak{g}=\bigoplus^{r}_{i=1}V_{i}$ satisfying (see e.g. \cite{fs3})
$$[V_{1},V_{j}]=V_{j+1}, \; 1\leq j\leq r-1;\;\; [V_{1},V_{r}]=\{0\}.$$
 As a simply connected
nilpotent group, $G$ is differential with $\mathbb{R}^{N}$,
$N=\sum^{r}_{i=1}\dim V_{i}$, via the
exponential map $\exp: \mathfrak{g}\rightarrow G$. There is a natural family
of nonisotropic dilations $\delta_{\lambda}: \mathfrak{g}\rightarrow\mathfrak{g}$ for  $\lambda>0$  and we define it as follows:
\begin{align*}
\delta_{\lambda}(X_{1}+X_{2}+\cdots+X_{r})=\lambda X_{1}+\lambda^{2}X_{2}+\cdots+\lambda^{r}X_{r}, \;\; X_{j}\in V_{j}, \; 1\leq j\leq r.
\end{align*}
The homogeneous dimension of $G$,  associated with $\delta_{\lambda}$, is   $Q=\sum_{j=1}^{r}j\dim V_{j}$.
Via the exponential map $\exp : \mathfrak{g} \rightarrow G$, we  define the
group of dilations on $G$ as follows:
\begin{align*}
\delta_{\lambda}(g)=\exp\circ \delta_{\lambda}\circ\exp^{-1}(g),\;\; g\in G.
\end{align*}

 Set $n_{j}=\dim V_{j}$, $1\leq
j\leq r$. Let $\{X_{1},\cdots, X_{n_{1}}\}$ be a basis of $V_{1}$
and denote by $\nabla_{G}=(X_{1},\cdots, X_{n_{1}})$ the horizontal gradient of $G$.
The sub-Laplacian on $G$ is
$\Delta_{G}=\sum_{i=1}^{n_{1}}X^{2}_{i}.$
The Sobolev space $W_{0}^{1,p}(G)$ is the  closure of $C_{0}^{\infty}(G)$
 with respect to the norm
\begin{align*}
\|u\|_{W_{0}^{1,p}(G)}=\left(\int_{G}|\nabla_{G}u|^{p}dg\right)^{\frac{1}{p}},
\end{align*}
where $dg$ is the Haar measure on $G$.
We remark that the Haar measure
on $G$,  induced by the exponential mapping from the Lebesgue
measure on $\mathfrak{g}=\mathbb{R}^{N}$,  coincides the Lebesgue
measure on $\mathbb{R}^{N}$.
The Folland-Stein inequality on $G$ reads that there exits some constant $C>0$ such that for each $u\in W_{0}^{1,p}(G)$ (see \cite{fs1,fs2}),
\begin{align}\label{1.1}
 \left(\int_{G}|u|^{\frac{pQ}{Q-p}}dg\right)^{\frac{Q-p}{pQ}}\leq C\left(\int_{G}|\nabla_{G}u|^{p}dg\right)^{\frac{1}{p}},\;\;1<p<Q.
\end{align}
For the existence and  regularity of minimizers  of the Folland-Stein
inequality (\ref{1.1}), we refer to \cite{va}.

The Heisenberg group is the simplest   example  of Carnot group of step 2.  We denote it by  $\mathbb{H}^{n}=(\mathbb{C}^{n}\times\mathbb{R},\circ)$.
The group law on $\mathbb{H}^{n}$ is given by
\begin{align*}
(z,t)\circ (z',t')=(z+z',t+t'+2\textrm{Im}z\cdot z'),
\end{align*}
where $z\cdot z'=\sum_{j=1}^{n}z_{j}\bar{z}_{j}'$.
The  homogeneous norm on $\mathbb{H}^{n}$ is given by
\begin{align*}
|(z,t)|=(|z|^{4}+t^{2})^{\frac{1}{4}}.
\end{align*}
In a series of  papers \cite{jl1,jl2,jl3},
 Jerison and Lee, among other results,  determined the explicit computation of the extremal functions in (\ref{1.1}) in the case $p=2$ and $G=\mathbb{H}^{n}$.
In fact, the extremal functions are, up to group translations and dilations,
\begin{align*}
c((1+|z|^{2})^{2}+t^{2})^{-\frac{Q-2}{4}},\;\; c\in\mathbb{R}.
\end{align*}
 Such inequalities play an important role in the  study of CR Yamabe problems.
Later, in a celebrated paper \cite{fl1}, Frank and  Lieb established sharp Hardy-Littlewood-Sobolev inequalities on $\mathbb{H}^{n}$.
We state the result  as follows:
\begin{theorem}[Frank-Lieb]\label{th1.1}
Let $0<\lambda<Q$ and $p=\frac{2Q}{Q-\lambda}$. Then for any $f,g\in L^{p}(\mathbb{H}^{n})$,
\begin{align*}
\left|\int\int_{\mathbb{H}^{n}\times \mathbb{H}^{n}}\frac{\overline{f(z,t)}g(z',t')}{|(z,t)^{-1}\circ (z',t')|^{\lambda}}dzdtdz'dt'\right|
\leq \left(\frac{\pi^{n+1}}{2^{n-1}n!}\right)^{\lambda/Q}\frac{n!\Gamma(\frac{Q-\lambda}{2})}{\Gamma^{2}(\frac{2Q-\lambda}{4})}\|f\|_{p}\|g\|_{p},
\end{align*}
with equality if and only if, up to group translations and dilations,
\begin{align*}
f=c((1+|z|^{2})^{2}+t^{2})^{-\frac{2Q-\lambda}{4}},\;\;g=c'((1+|z|^{2})^{2}+t^{2})^{-\frac{2Q-\lambda}{4}}
\end{align*}
for some $c,c'\in\mathbb{C}$.
\end{theorem}
In particular, choosing $\lambda=Q-2$ in Theorem \ref{th1.1} yields the Jerison and Lee's inequality. Using the method in \cite{fl1},  Frank and Lieb \cite{fl2}
  also gave a new, rearrangement-free proof of
sharp Hardy-Littlewood-Sobolev inequalities  on $\mathbb{R}^{n}$.
Recently, Hang and Wang \cite{hang} present a shorter proof of the Frank-Lieb inequality, in  which
they bypasses the subtle proof of existence and the Hersch-type argument via subcritical approximation.

Some of the results of Theorem \ref{th1.1}  have been generalized to the cases  of  quaternionic Heisenberg group and octonionic Heisenberg group (see \cite{chz1,chz2,imv1,imv2}).
We note that Heisenberg group,  quaternionic Heisenberg group and octonionic Heisenberg group are known as the groups of Iwasawa type, i.e.,
the nilpotent component in the
Iwasawa decomposition of simple groups of rank one (see e.g. \cite{cdk}).

The aim of  this paper is to look for the  optimal constant of (\ref{1.1}) when $p=2$ and $G$ is a group of Heisenberg type (in short a H-type group).
Recall that
a H-type group $G$ is a
Carnot group of step two with the following properties (see Kaplan \cite{k}): the Lie
algebra $\mathfrak g$ of $G$ is endowed with an inner product $\langle,\rangle$
such that, if $\mathfrak z$ is the center of $\mathfrak g$, then
$[\mathfrak z^{\perp}, \mathfrak z^{\perp}]=\mathfrak z$ and
moreover, for every fixed $z\in \mathfrak z$, the map $J_{z}:
\mathfrak z^{\perp}\rightarrow \mathfrak z^{\perp}$ defined by
\begin{eqnarray}\label{1.2}
\langle J_{z}(v),\omega\rangle=\langle z,[v,\omega]\rangle,\;\;\forall \omega\in\mathfrak
z^{\perp}
\end{eqnarray}
is an orthogonal map whenever $\langle z,z\rangle=1$.
It is known (see \cite{cdk}) that a H-type group $G$ is the  group of Iwasawa  type if and only if its Lie algebra satisfies the following $J^{2}$-condition:
for any $v\in \mathfrak{z}^{\perp}$  and $z,z'\in \mathfrak{z}$ such that $\langle z,z'\rangle=0$,  there exists $z''\in \mathfrak{z}$
such that
\begin{align*}
  J_{z}J_{z'}v=J_{z''}v.
\end{align*}
Therefore,  most of H-type groups are not groups of Iwasawa  type.

  Set $m=\dim \mathfrak z^{\perp}$ and $n=\dim \mathfrak z$. Since $G$ has step two, we
can fix on $G$ a system of coordinates $(x,t)$ such that the group
law on $G$ has the form (see \cite{bu})
\begin{equation}\label{1.3}
\begin{split}
(x,t)\circ (x',t')=\begin{pmatrix}
  x_{i}+x'_{i},\;\; i=1,2,\cdots, m \\
  t_{j}+t'_{j}+\frac{1}{2}\langle x,U^{(j)}x'\rangle,\;\;j=1,2,\cdots, n\\
\end{pmatrix}
\end{split}
\end{equation}
for suitable skew-symmetric matrices $U^{(j)}$'s. Nextly,  we
 set
\begin{align}\label{2.2}
U(\xi)=&\left[\left(1+\frac{|x|^{2}}{4}\right)^{2}+|t|^{2}\right]^{-\frac{Q-2}{4}},\;\xi=(x,t)\in
G;\\
U_{\lambda,\eta}(\xi)=&\lambda^{\frac{Q-2}{2}}U(\delta_{\lambda}(\eta^{-1}\circ\xi)),\;\;\;\;\;\;\;\;\;\;\;\;\; \eta\in G. \label{2.3}
\end{align}
It has been shown that $\big[\frac{m(Q-2)}{4}\big]^{\frac{Q-2}{4}}U_{\lambda,\eta}(\xi)$ satisfies the  Yamabe-type equation (see \cite{gv1,gv2})
\begin{eqnarray*}
\Delta_{G}\Big[\frac{m(Q-2)}{4}\Big]^{\frac{Q-2}{4}}U_{\lambda,\eta}+\left\{\Big[\frac{m(Q-2)}{4}\Big]^{\frac{Q-2}{4}}U_{\lambda,\eta}\right\}^{\frac{Q+2}{Q-2}}=0,
\end{eqnarray*}
or equivalently,
\begin{align}\label{2.4}
\Delta_{G}U_{\lambda,\eta}+\frac{m(Q-2)}{4}U_{\lambda,\eta}^{\frac{Q+2}{Q-2}}=0.
\end{align}
In the paper \cite{gv1}, Garofalo and  Vassilev gave the
following conjecture:
\vspace{0.1cm}

\textbf{Conjecture}\;(Garofalo-Vassilev).
In a H-type group $G$, the functions  $\big[\frac{m(Q-2)}{4}\big]^{\frac{Q-2}{4}}U_{\lambda,\eta}(\xi)$  are the only nontrivial entire
solutions to
\begin{align*}
\left\{
  \begin{array}{ll}
 \Delta_{G}u+u^{\frac{Q+2}{Q-2}}=0   ,  \\
  u\in W^{1,2}_{0}(G), u\geq0.
  \end{array}
\right.
\end{align*}
\vspace{0.1cm}

If the conjecture is true, then one can obtain the  optimal constant of $L^{2}$ Folland-Stein inequality on H-type groups.  In  this paper we shall
use the method given by Frank and Lieb \cite{fl1,fl2} and Hang and Wang \cite{hang} to determine the optimal constant, instead of proving the conjecture directly.
To this end, we have
\begin{theorem}\label{th1.2}
It holds that
\begin{align}\label{1.4}
\int_{G}|\nabla_{G}u|^{2}dxdt\geq S_{m,n} \left(\int_{G}|u|^{\frac{2Q}{Q-2}}dxdt\right)^{\frac{Q-2}{Q}}, \; u\in W^{1,2}_{0}(G),
\end{align}
where
\begin{align*}
S_{m,n}=4^{-\frac{2n}{Q}}m(Q-2)\pi^{\frac{m+n}{Q}}\left(\frac{\Gamma(\frac{m+n}{2})}{\Gamma(m+n)}\right)^{2/Q}.
\end{align*}
The inequality is sharp and an extremal function  is
\begin{align*}
U(x,t)=\left[\left(1+\frac{|x|^{2}}{4}\right)^{2}+|t|^{2}\right]^{-\frac{Q-2}{4}}.
\end{align*}
\end{theorem}

By Theorem \ref{th1.2}, it is easy to see that the functions $cU_{\lambda,\eta}(\xi)(c\in\mathbb{R})$ are also  extremal functions of inequality (\ref{1.4}).

As an application of Theorem \ref{th1.2}, we
 study  the
eigenvalues of
\begin{eqnarray}\label{1.8}
-\Delta_{G}v=\mu U_{\lambda,\eta}^{\frac{4}{Q-2}}v,\;\;v\in
W^{1,2}_{0}(G).
\end{eqnarray}
We note that  the eigenvalues
of (\ref{1.8}) play  an important role in the study of   stability  for the  Folland-Stein inequality
(see  \cite{be,bl,ch,do,lu} for the case of Euclidean space and \cite{lo} for the case of Heisenberg group).
In Lemma \ref{lm3.2} we show that the  embedding map $W^{1,2}_{0}(G)\hookrightarrow L^{2}(G,
U(x,t)^{\frac{4}{Q-2}}dxdt)$ is compact. So the spectrum of (\ref{1.8}) is discrete (we note that
the spectrum  do not depend on $\lambda$ and
$\eta$, see Theorem \ref{th1.3}). Furthermore, we have  the following theorem:
\begin{theorem}\label{th1.3}
 Let $\mu_{i},\;i=1,2,\cdots$ be the eigenvalues
of (\ref{1.8}) given in increasing order. Then

(1) $\mu_{1}=\frac{m(Q-2)}{4}$ is simple with eigenfunction
$U_{\lambda,\eta}$.

(2) $\mu_{2}=\frac{m(Q+2)}{4}$ and
$\{\partial_{\lambda}U_{\lambda,\eta},\;\nabla_{\eta}U_{\lambda,\eta}\}$ are eigenfunctions.

Furthermore, the eigenvalues do not depend on $\lambda$ and
$\eta$.
\end{theorem}
\begin{remark}
It seems that
 $\mu_{2}$ has multiplicity
$m+n+1$ with corresponding eigenspace spanned by
$\{\partial_{\lambda}U_{\lambda,\eta},\;\nabla_{\eta}U_{\lambda,\eta}\}$.  However, we fail to prove it.
Once it has been  proven, it would provide a generalization of the results of Bianchi and  Egnell (\cite{be}, Lemma A1)
to the setting of H-type groups.
\end{remark}

\section{preliminaries on H-type groups}

In the rest of paper, we let $G$ be a H-type group with group law given by (\ref{1.3}).
The nonisotropic dilations $\delta_{\lambda}$ on $G$ is
\begin{align*}
 \delta_{\lambda}(x,t)=(\lambda x,\lambda^{2}t).
\end{align*}
For $(x,t)\in G$, the  homogeneous norm of $(x,t)$ is
\begin{align*}
\rho(x,t)=\left(\frac{|x|^{4}}{16}+|t|^{2}\right)^{\frac{1}{4}}.
\end{align*}
With this norm $\rho$, we can define the ball centered at origin with radius $R$
\begin{align*}
 B_{R}(0)=\{(x,t)\in G: \rho(x,t)<R\}
\end{align*}
and the unit sphere $\Sigma=\partial
B_{1}(0)=\{(x,t)\in G: \rho(x,t)=1\}$.

 Given any
$(x,t)\in G$ with $\rho(x,t)\neq0$, we set $x^*=\frac{x}{\rho(x,t)}$
and $t^*=\frac{t}{\rho(x,t)^2}$.
 The polar coordinates on
$G$  associated with $\rho$ are the following (see
\cite{fs3}):
$$\int_{G}f(x,t)dxdt=
\int^{\infty}_{0}\int_{\Sigma}f(
\rho x^{\ast},\rho^{2}t^{\ast})\rho^{Q-1}d\sigma d\rho,\;\;f\in
L^{1}(G).$$

The following theorem was proved in \cite{bu}, Theorem A.2:
\begin{theorem}
\label{propmain1} $G$ is a H-type group if and only if $G$ is
(isomorphic to) $\mathbb{R}^{m+n}$ with the group law in (\ref{1.3}) and
the matrices $U^{(1)}, U^{(2)}, \cdots, U^{(n)}$ have the following
properties:

(1) $U^{(j)}$ is a $m\times m$ skew symmetric and orthogonal matrix,
for every $j=1,2,\cdots, n$;

(2) $U^{(i)} U^{(j)}+U^{(j)}U^{(i)}=0$ for every $i,j\in
\{1,2,\cdots, n\}$ with $i\neq j$.
\end{theorem}

The vector field in the Lie algebra
$\mathfrak{g}$  that agrees at the
origin with $\frac{\partial}{\partial x_{j}}(j=1,\cdots,m)$ is given
by
\begin{equation*}
\begin{split}
X_{j} = \frac{\partial}{\partial
x_{j}}+\frac{1}{2}\sum^{n}_{k=1}\left(
\sum_{i=1}^{m}U^{(k)}_{i,j}x_{i}\right)\frac{\partial}{\partial
t_{k}}
\end{split}
\end{equation*}
and $\mathfrak{g}$ is spanned by the left-invariant vector
fields $X_{1},\cdots,X_{m},T_{1}=\frac{\partial}{\partial
t_{1}},\cdots,T_{n}=\frac{\partial}{\partial t_{n}}$.
Furthermore (see \cite{bu}, Page 200, (A.4) ),
\begin{align}\label{bb2.1}
[X_{i},X_{j}]=\sum^{n}_{r=1}U^{(r)}_{i,j}T_{r},\; i,j\in
\{1,2,\cdots, n\}.
\end{align}
The exponential map $\exp: \mathfrak{g}\rightarrow G$ is
\begin{align*}
\exp: \mathfrak{g}\rightarrow \mathbb{R}^{m+n},\;\; \sum_{i=1}^{m}x_{i}X_{i}+ \sum_{j=1}^{n}t_{j}T_{j}\mapsto (x,t).
\end{align*}
We note that by  exponential mapping,   the group law (\ref{1.3}) is nothing  but the
Baker-Campbell-Hausdorff formula (see \cite{bu}, the proof of Theorem A.2)
\begin{align*}
\exp X \circ \exp Y=\exp(X+Y+\frac{1}{2}[X,Y]), \; X, Y\in \mathfrak{g}.
\end{align*}
Using (\ref{bb2.1}), we have that for $t=(t_{1},\cdots,t_{n})=t_{1}T_{1}+\cdots+
t_{n}T_{n}$ and $x=(x_{1},\cdots,x_{m})=x_{1}X_{1}+\cdots+
x_{m}X_{m}$, the map $J_{t}$, defined by (\ref{1.2}),  is (see also \cite{bu}, Page 201)
\begin{equation*}
  \begin{split}
 J_{t}x=&\sum^{n}_{r=1}\sum^{m}_{i=1}t_{r}x_{i}J_{T_{r}}(X_{i})=\sum^{n}_{r=1}\sum^{m}_{i=1}t_{r}x_{i}
\left(\sum^{m}_{j=1}U^{(r)}_{i,j}X_{j}\right)\\
=&\sum^{m}_{j=1}\left(\sum^{n}_{r=1}\sum^{m}_{i=1}t_{r}x_{i}U^{(r)}_{i,j}\right)X_{j}.
  \end{split}
\end{equation*}
Since $J_{t}$ is an orthogonal map whenever $|t|=1$, we obtain
\begin{align}\label{2.1}
|J_{t}x|^{2}=|t|^{2}|x|^{2}=\sum^{m}_{j=1}\left(\sum^{n}_{r=1}\sum^{m}_{i=1}t_{r}x_{i}U^{(r)}_{i,j}\right)^{2}.
\end{align}

The horizontal gradient on $G$ is $\nabla_{G}=(X_{1},\cdots,X_{m})$.
The
sub-Laplacian on $G$ is given by (see \cite{bu}, Remark A.6.)
\begin{equation*}
\begin{split}
\Delta_{G}=\sum^{m}_{j=1}X^{2}_{j}&=\sum^{m}_{j=1}\left(\frac{\partial}{\partial
x_{j}}+\frac{1}{2}\sum^{n}_{k=1}\left(
\sum_{i=1}^{m}U^{(k)}_{i,j}x_{i}\right)\frac{\partial}{\partial
t_{k}}\right)^{2}\\
&=\Delta_{x}+\frac{1}{4}|x|^{2}\Delta_{t}+\sum^{n}_{k=1}\langle x,U^{(k)}\nabla_{x}\rangle\frac{\partial}{\partial
t_{k}},
\end{split}
\end{equation*}
where
$$\Delta_{x}=\sum^{m}_{j=1}\left(\frac{\partial}{\partial
x_{j}}\right)^{2},\;\;\Delta_{t}=\sum^{n}_{k=1}\left(\frac{\partial}{\partial
t_{k}}\right)^{2}.$$
We remark that $\Delta_{G}$ is homogeneous of degree two with respect to $\delta_{\lambda}$.

By using (\ref{2.4}), we have the following Hardy inequality (see \cite{rt}, Corollary 1.4  for Hardy inequality of fractional powers of the sublaplacian on $G$).
\begin{lemma}\label{lm2.1}
It holds that, for $u\in W^{1,2}_{0}(G)$,
\begin{align*}
 \int_{G}|\nabla_{G}u|^{2}dxdt\geq \frac{m(Q-2)}{4}\int_{G}\frac{u^{2}}{(1+\frac{|x|^{2}}{4})^{2}+|t|^{2}}dxdt,
\end{align*}
with equality if and only if $u=cU(x,t)$, where  $c\in \mathbb{R}$ and $U(x,t)$ is given by (\ref{2.2}).
\end{lemma}
\begin{proof}
We have, through integration by parts,
\begin{align}
0\leq&\int_{G}U^{2}|\nabla_{G}(U(x,t)^{-1}u)|^{2}dxdt\nonumber\\
=&\int_{G}\left|\nabla_{G}u-\frac{u}{U}\nabla_{G}U\right|^{2}dxdt\nonumber\\
=&\int_{G}|\nabla_{G}u|^{2}dxdt+\int_{G}\frac{|\nabla_{G} U|^{2}}{U^{2}}u^{2}dxdt-
\int_{G}\frac{1}{U}\langle \nabla_{G}u^{2},\nabla_{G}U\rangle dxdt\nonumber\\
=&\int_{G}|\nabla_{G}u|^{2}dxdt+
\int_{G}u^{2}\frac{1}{U}\Delta_{G}U dxdt\nonumber\\
=&\int_{G}|\nabla_{G}u|^{2}dxdt- \frac{m(Q-2)}{4}\int_{G}\frac{u^{2}}{(1+\frac{|x|^{2}}{4})^{2}+|t|^{2}}dxdt. \label{b2.1}
\end{align}
To get the last equality, we use (\ref{2.4}). The desired result follows.
\end{proof}

Set $\eta=(y_{1},\cdots,y_{m},w_{1},\cdots,w_{n})\in G$.
By (\ref{2.4}), we have
\begin{equation}\label{2.5}
  \begin{split}
  &\Delta_{G}\frac{\partial U_{\lambda,\eta}}{\partial
  y_{j}}+\frac{m(Q+2)}{4}U_{\lambda,\eta}^{\frac{4}{Q-2}}\frac{\partial U_{\lambda,\eta}}{\partial
  y_{j}}=0,\;\;\;\;j=1,,\cdots,m,\\
&\Delta_{G}\frac{\partial U_{\lambda,\eta}}{\partial
  w_{r}}+\frac{m(Q+2)}{4}U_{\lambda,\eta}^{\frac{4}{Q-2}}\frac{\partial U_{\lambda,\eta}}{\partial
  w_{r}}=0,\;\;\;\;r=1,,\cdots,n,\\
&\Delta_{G}\frac{\partial U_{\lambda,\eta}}{\partial
  \lambda}+\frac{m(Q+2)}{4}U_{\lambda,\eta}^{\frac{4}{Q-2}}\frac{\partial U_{\lambda,\eta}}{\partial
  \lambda}=0.
  \end{split}
\end{equation}
Furthermore, we have the following lemma:
\begin{lemma}\label{lm2.2}
It holds that
\begin{align*}
\sum_{j=1}^{m}\left|\frac{\partial U_{\lambda,\eta}}{\partial
  y_{j}}|_{\lambda=1,\eta=0}\right|^{2}+\sum_{r=1}^{n}\left|\frac{\partial U_{\lambda,\eta}}{\partial
  w_{r}}|_{\lambda=1,\eta=0}\right|^{2}+\frac{1}{4}\left|\frac{\partial U_{\lambda,\eta}}{\partial
 \lambda}|_{\lambda=1,\eta=0}\right|^{2}=\frac{(Q-2)^{2}}{16} U(\xi)^{2}.
\end{align*}
\end{lemma}
\begin{proof}
It is easy to see $\eta^{-1}=-\eta$. Therefore, by (\ref{1.3}) and (\ref{2.3}), we have
\begin{align*}
U_{\lambda,\eta}(x,t)=&\lambda^{\frac{Q-2}{2}}\left[\left(1+\frac{\lambda^{2}}{4}\sum^{m}_{i=1}(x_{i}-y_{i})^{2}\right)^{2}+
\lambda^{4}\sum^{n}_{r=1}\left(t_{r}-w_{r}-\frac{\langle y,U^{(r)}x\rangle}{2}\right)^{2}\right]^{-\frac{Q-2}{4}}.
\end{align*}
We compute
\begin{align*}
\frac{\partial U_{\lambda,\eta}}{\partial
  y_{j}}|_{\lambda=1,\eta=0}=&-\frac{Q-2}{4}U(\xi)^{\frac{Q+2}{Q-2}}
  \left[2\left(1+\frac{|x|^{2}}{4}\right)\cdot\left(-\frac{x_{j}}{2}\right)+\sum^{n}_{r=1}t_{r}\left(-
  \sum^{m}_{i=1}U_{j,i}^{(r)}x_{i}\right)\right]\\
  =&\frac{Q-2}{4} U(\xi)^{\frac{Q+2}{Q-2}}\left[\left(1+\frac{|x|^{2}}{4}\right)x_{j}+\sum^{n}_{r=1}
  \sum^{m}_{i=1}t_{r}x_{i}U_{j,i}^{(r)}\right],\;\;j=1,\cdots,m;\\
  \frac{\partial U_{\lambda,\eta}}{\partial
  w_{r}}|_{\lambda=1,\eta=0}=&-\frac{Q-2}{4} U(\xi)^{\frac{Q+2}{Q-2}}(-2t_{r})\\
  =&\frac{Q-2}{2}U(\xi)^{\frac{Q+2}{Q-2}}t_{r}, \;\;\;\;\;\;\;\;r=1,\cdots,n;\\
 \frac{\partial U_{\lambda,\eta}}{\partial
  \lambda}|_{\lambda=1,\eta=0}=&\frac{Q-2}{2}U(\xi)-\frac{Q-2}{4}
 U(\xi)^{\frac{Q+2}{Q-2}}\left(2\left(1+\frac{|x|^{2}}{4}\right)\cdot
  \frac{|x|^{2}}{2}+4|t|^{2}\right)\\
  =&-\frac{Q-2}{2}U(\xi)^{\frac{Q+2}{Q-2}}\left(-1+\frac{|x|^{4}}{16}+|t|^{2}\right).
\end{align*}
Since each $U^{(j)}(1\leq j\leq n)$ is a $m\times m$ skew symmetric matrix,  we have, by using (\ref{2.1}),
\begin{align*}
\sum_{j=1}^{m}\left|\frac{\partial U_{\lambda,\eta}}{\partial
  y_{j}}|_{\lambda=1,\eta=0}\right|^{2}=&\frac{(Q-2)^{2}}{16}U(\xi)^{2\frac{Q+2}{Q-2}}\sum_{j=1}^{m}\left[\left(1+\frac{|x|^{2}}{4}\right)x_{j}-\sum^{n}_{r=1}
  \sum^{m}_{i=1}t_{r}x_{i}U_{i,j}^{(r)}\right]^{2}\\
  =&\frac{(Q-2)^{2}}{16} U(\xi)^{2\frac{Q+2}{Q-2}}\left[
  \left(1+\frac{|x|^{2}}{4}\right)^{2}|x|^{2}+|t|^{2}|x|^{2}-\right.\\
  &\left. 2\left(1+\frac{|x|^{2}}{4}\right)\sum^{n}_{r=1}t_{r} \left(\sum_{i=1}^{m}\sum^{m}_{j=1}U_{i,j}^{(r)}x_{i}x_{j}\right)\right]\\
  =&\frac{(Q-2)^{2}}{16} U(\xi)^{2\frac{Q+2}{Q-2}}\left[
  \left(1+\frac{|x|^{2}}{4}\right)^{2}|x|^{2}+|t|^{2}|x|^{2}\right].
\end{align*}
To get the last equality, we use the fact
\begin{align*}
\sum_{i=1}^{m}\sum^{m}_{j=1}U_{i,j}^{(r)}x_{i}x_{j}=0
\end{align*}
since  $U^{(r)}(1\leq r\leq n)$ is a $m\times m$ skew symmetric matrix.
Therefore, we have
\begin{align*}
&\sum_{j=1}^{m}\left|\frac{\partial U_{\lambda,\eta}}{\partial
  y_{j}}|_{\lambda=1,\eta=0}\right|^{2}+\sum_{r=1}^{n}\left|\frac{\partial U_{\lambda,\eta}}{\partial
  w_{r}}|_{\lambda=1,\eta=0}\right|^{2}+\frac{1}{4}\left|\frac{\partial U_{\lambda,\eta}}{\partial
 \lambda}|_{\lambda=1,\eta=0}\right|^{2}\\
 =&\frac{(Q-2)^{2}}{16} U(\xi)^{2\frac{Q+2}{Q-2}}\left[
  \left(1+\frac{|x|^{2}}{4}\right)^{2}|x|^{2}+|t|^{2}|x|^{2}\right]+\frac{(Q-2)^{2}}{4} U(\xi)^{2\frac{Q+2}{Q-2}}|t|^{2}+\\
  &\frac{(Q-2)^{2}}{16} U(\xi)^{2\frac{Q+2}{Q-2}}\left(-1+\frac{|x|^{4}}{16}+|t|^{2}\right)^{2}   \\
 =&\frac{(Q-2)^{2}}{16} U(\xi)^{2\frac{Q+2}{Q-2}}\left[
   \left(1+\frac{|x|^{2}}{4}\right)^{2}|x|^{2}+|t|^{2}|x|^{2}+
   4|t|^{2}+\left(-1+\frac{|x|^{4}}{16}+|t|^{2}\right)^{2}\right]\\
   =&\frac{(Q-2)^{2}}{16} U(\xi)^{2\frac{Q+2}{Q-2}}
   \left[\left(1+\frac{|x|^{2}}{4}\right)^{2}+|t|^{2}\right]^{2}\\
   =&\frac{(Q-2)^{2}}{16} U(\xi)^{2}.
\end{align*}
To get the third  equality, we use the fact
\begin{align*}
 &\left(-1+\frac{|x|^{4}}{16}+|t|^{2}\right)^{2}\\
 =&\left[\left(1+\frac{|x|^{2}}{4}\right)^{2}+|t|^{2}-2\left(1+\frac{|x|^{2}}{4}\right)\right]^{2}\\
 =& \left[\left(1+\frac{|x|^{2}}{4}\right)^{2}+|t|^{2}\right]^{2}+4\left(1+\frac{|x|^{2}}{4}\right)^{2}-
 4\left(1+\frac{|x|^{2}}{4}\right)\left[\left(1+\frac{|x|^{2}}{4}\right)^{2}+|t|^{2}\right]\\
 =&\left[\left(1+\frac{|x|^{2}}{4}\right)^{2}+|t|^{2}\right]^{2}-
  \left(1+\frac{|x|^{2}}{4}\right)^{2}|x|^{2}-|t|^{2}|x|^{2}-
   4|t|^{2}.
\end{align*}
This completes the proof of Lemma \ref{lm2.2}.
\end{proof}

For simplicity, we set
\begin{equation}\label{2.6}
  \begin{split}
 \omega_{j}=&\frac{4}{Q-2}U(\xi)^{-1}\frac{\partial U_{\lambda,\eta}}{\partial
  y_{j}}|_{\lambda=1,\eta=0},\;\;j=1,\cdots,m;\\
\omega_{j+r}=&\frac{4}{Q-2}U(\xi)^{-1}\frac{\partial U_{\lambda,\eta}}{\partial
  w_{r}}|_{\lambda=1,\eta=0},\;\;r=1,\cdots,n;  \\
\omega_{m+n+1}=&\frac{2}{Q-2}U(\xi)^{-1}  \frac{\partial U_{\lambda,\eta}}{\partial
 \lambda}|_{\lambda=1,\eta=0}.
  \end{split}
\end{equation}
By Lemma \ref{lm2.2} and (\ref{2.5}), we have
\begin{align}\label{2.7}
\sum_{j=1}^{m+n+1}\omega_{j}^{2}=&1;\\
\label{2.8}
\Delta_{G}(U(\xi)\omega_{j})+\frac{m(Q+2)}{4}U(\xi)^{\frac{Q+2}{Q-2}}\omega_{j}=&0,\;
1\leq j\leq m+n+1.
\end{align}

\section{Proof of Theorem \ref{th1.2} and \ref{th1.3}}
In this section, we shall prove  Theorem \ref{th1.2} and \ref{th1.3}. The proof  depends on a
scheme of subcritical approximation due to Hang and Wang \cite{hang}.
We first establish the following  subcritical Sobolev inequality on $G$.
\begin{lemma}\label{lm3.1}
Let $2\leq p<\frac{2Q}{Q-2}$. There exists $C>0$ such that for each $u\in W^{1,2}_{0}(G)$,
\begin{align*}
\int_{G}|\nabla_{G}u|^{2}dxdt\geq C\left(\int_{G}|u|^{p}U(x,t)^{\frac{2Q}{Q-2}-p}dxdt\right)^{\frac{2}{p}}.
\end{align*}
\end{lemma}
\begin{proof}
By H\"older's inequality, we have
\begin{align*}
&\int_{G}|u|^{p}U(x,t)^{\frac{2Q}{Q-2}-p}dxdt\\
=&\int_{G}\left(|u|U(x,t)^{\frac{2}{Q-2}}\right)^{Q-\frac{Q-2}{2}p}|u|^{\frac{Q}{2}(p-2)}dxdt\\
 \leq& \left(\int_{G}|u|^{2}U(x,t)^{\frac{4}{Q-2}}dxdt\right)^{\frac{2Q-(Q-2)p}{4}}
 \left(\int_{G}|u|^{\frac{2Q}{Q-2}}dxdt\right)^{\frac{(Q-2)(p-2)}{4}}\\
 =&\left(\int_{G}\frac{u^{2}}{(1+\frac{|x|^{2}}{4})^{2}+|t|^{2}}dxdt\right)^{\frac{2Q-(Q-2)p}{4}}
 \left(\int_{G}|u|^{\frac{2Q}{Q-2}}dxdt\right)^{\frac{(Q-2)(p-2)}{4}}\\
 \leq& C\int_{G}|\nabla_{G}u|^{2}dxdt,
\end{align*}
where $C$ is a positive constant independent of $u$.
To get the last inequality above, we use Folland-Stein
inequality (\ref{1.1}) and Lemma \ref{lm2.1}. This completes the proof of Lemma \ref{lm3.1}.
\end{proof}

By Lemma \ref{lm3.1}, we have   $W^{1,2}_{0}(G)\hookrightarrow L^{p}(G,
U(x,t)^{\frac{2Q}{Q-2}-p}dxdt).$
Furthermore, the
embedding map is compact.
\begin{lemma}\label{lm3.2}
Let $2\leq p<\frac{2Q}{Q-2}$. The embedding map $W^{1,2}_{0}(G)\hookrightarrow L^{p}(G,
U(x,t)^{\frac{2Q}{Q-2}-p}dxdt)$ is compact.
\end{lemma}
\begin{proof}
The proof is similar to that given by  Schneider
(see \cite{s}, section 2.2).

Let $\phi : G\rightarrow [0,1]$ be a cut-off function
that is equal to one in $B_{1}(0)$ and zero outside of $B_{2}(0)$.
Consider the operator
\[
I_{R}: W^{1,2}_{0}(G)\hookrightarrow
L^{p}(G,
U(x,t)^{\frac{2Q}{Q-2}-p}dxdt)
\]
defined by
$I_{R}(u)=u(x,t)\phi\left(\frac{x}{R},\frac{t}{R^{2}}\right)$.
Since
the imbedding map $W^{1,2}_{0}(B_{2}(0))\hookrightarrow L^{p}(B_{2}(0))$
is compact, so is  $I_{R}$. Moreover, by H\"older's inequality,
\begin{equation*}
\begin{split}
&\int_{G}|u-I_{R}(u)|^{p}U(x,t)^{\frac{2Q}{Q-2}-p}dxdt\\
\leq& \int_{G\setminus
B_{R}(0)}|u|^{p}U(x,t)^{\frac{2Q}{Q-2}-p}dxdt\\
\leq&\left(\int_{G\setminus
B_{R}(0)}|u|^{\frac{2Q}{Q-2}}dxdt\right)^{\frac{Q-2}{2Q}p}\left(\int_{G\setminus
B_{R}(0)}U(x,t)^{\frac{2Q}{Q-2}}dxdt\right)^{1-\frac{Q-2}{2Q}p}\\
\leq& C\left(\int_{G}|\nabla_{G}u|^{2}dxdt\right)^{\frac{p}{2}}\left(\int_{G\setminus
B_{R}(0)}U(x,t)^{\frac{2Q}{Q-2}}dxdt\right)^{1-\frac{Q-2}{2Q}p}.
\end{split}
\end{equation*}
To get the last inequality above, we use the Folland-Stein
inequality (\ref{1.1}).
By polar coordinates, we have
\begin{align*}
\int_{G\setminus
B_{R}(0)}U(x,t)^{\frac{2Q}{Q-2}}dxdt=&\int_{G\setminus B_{R}(0)}\left[\left(1+\frac{|x|^{2}}{4}\right)^{2}+|t|^{2}\right]^{-\frac{Q}{2}}dxdt\\
\leq& \int_{G\setminus B_{R}(0)}\frac{1}{(\frac{|x|^{2}}{4}+|t|^{2})^{\frac{Q}{2}}}dxdt\\
=&\int^{\infty}_{R}\int_{\Sigma}\frac{1}{\rho^{2Q}}\rho^{Q-1}d\rho d\sigma\\
=&|\Sigma|\frac{1}{QR^{Q}}\rightarrow0,\;\;R\rightarrow\infty,
\end{align*}
where $|\Sigma|$ is the volume of $\Sigma$. Therefore,
 the embedding map $$W^{1,2}_{0}(G)\hookrightarrow L^{p}(G,
U(x,t)^{\frac{2Q}{Q-2}-p}dxdt)$$ is  a limit of
compact operators and thus it is compact.
The proof of Lemma \ref{lm3.2} is thereby completed.
\end{proof}

By Lemma \ref{lm3.2}, the
minimization problem
\begin{align}\label{3.1}
 \Lambda_{p}=\inf\left\{\int_{G}|\nabla_{G}u|^{2}dxdt : \; \int_{G}|u|^{p}U(x,t)^{\frac{2Q}{Q-2}-p}dxdt=1\right\}, \; 2\leq p<\frac{2Q}{Q-2},
\end{align}
has a  positive solution $u$.
We shall show that such $u$ satisfies
a moment zero condition. The main result is the following lemma:
\begin{lemma}\label{lm3.3}
Let $2\leq p<\frac{2Q}{Q-2}$ and $u$ be a  positive solution of (\ref{3.1}).
Then we have
\begin{align}\label{3.2}
\int_{G}u^{p}U(x,t)^{\frac{2Q}{Q-2}-p}\omega_{i}dxdt=0,\;i=1,2,\cdots,m+n+1,
\end{align}
where $\omega_{i}(1\leq i\leq m+n+1)$ is given by (\ref{2.6}).
\end{lemma}
\begin{proof}
For simplicity, we set
\begin{align*}
\mathcal{F}_{p}(u)=\frac{\int_{G}|\nabla_{G}u|^{2}dxdt}{\left(\int_{G}|u|^{p}U(x,t)^{\frac{2Q}{Q-2}-p}dxdt\right)^{\frac{2}{p}}}
\end{align*}
and
\begin{align*}
u_{\lambda^{-1},\eta^{-1}}(\xi)=&\lambda^{-\frac{Q-2}{2}}u(\delta_{\lambda^{-1}}(\eta\circ\xi)),\; \lambda>0,\;\eta=(y_{1},\cdots,y_{m},w_{1},\cdots,w_{n})\in G.
\end{align*}
A simple calculation shows
\begin{align*}
 \int_{G}|\nabla_{G} u_{\lambda^{-1},\eta^{-1}}|^{2}dxdt=&\int_{G}|\nabla_{G}u|^{2}dxdt;\\
 \int_{G}u_{\lambda^{-1},\eta^{-1}}^{p}U(x,t)^{\frac{2Q}{Q-2}-p}dxdt=&\int_{G}u^{p}U_{\lambda,\eta}(x,t)^{\frac{2Q}{Q-2}-p}dxdt,
\end{align*}
where $U_{\lambda,\eta}$ is given by (\ref{2.3}). Therefore,
\begin{align}\label{3.3}
\mathcal{F}_{p}(u_{\lambda^{-1},\eta^{-1}}(\xi))=
\frac{\int_{G}|\nabla_{G}u|^{2}dxdt}{\left(\int_{G}|u|^{p}U_{\lambda,\eta}(x,t)^{\frac{2Q}{Q-2}-p}dxdt\right)^{\frac{2}{p}}}.
\end{align}
Since $u$ is a  positive solution of (\ref{3.1}), we have
\begin{equation}\label{3.4}
  \begin{split}
    \frac{\partial}{\partial y_{j}}\mathcal{F}_{p}(u_{\lambda^{-1},\eta^{-1}}(\xi))|_{\lambda=1,\eta=0}=&0,\; j=1,\cdots,m;\\
    \frac{\partial}{\partial w_{r}}\mathcal{F}_{p}(u_{\lambda^{-1},\eta^{-1}}(\xi))|_{\lambda=1,\eta=0}=&0,\; r=1,\cdots,n;\\
    \frac{\partial}{\partial \lambda}\mathcal{F}_{p}(u_{\lambda^{-1},\eta^{-1}}(\xi))|_{\lambda=1,\eta=0}=&0.
  \end{split}
\end{equation}
Combining  (\ref{3.3}) and (\ref{3.4}) yields (\ref{3.2}).
This completes the proof of Lemma \ref{lm3.3}.
\end{proof}

\begin{remark}
We remark that Lemma \ref{lm3.3} is also valid for $u>0$  satisfying the  Yamabe-type equation
\begin{align*}
\Delta_{G}u+\Lambda_{p}u^{p-1}U(x,t)^{\frac{2Q}{Q-2}-p}=0,\;\; \int_{G}|\nabla_{G}u|^{2}dxdt<\infty.
\end{align*}
The proof is same and we omit it (see \cite{hang}, Corollary 1  for the case of CR sphere).
\end{remark}

\begin{lemma}\label{lm3.4}
It holds that, for any $u\in W^{1,2}_{0}(G)$,
\begin{align*}
\sum_{i=1}^{m+n+1}\int_{G}|\nabla_{G}(u\omega_{i})|^{2}dxdt=\int_{G}|\nabla_{G}u|^{2}dxdt+m
\int_{G}\frac{u^{2}}{(1+\frac{|x|^{2}}{4})^{2}+|t|^{2}}dxdt.
\end{align*}
\end{lemma}
\begin{proof}
 Let $u=U(\xi)v$. We compute, through integration by parts,
\begin{align}
 \sum_{i=1}^{m+n+1}\int_{G}|\nabla_{G}(u\omega_{i})|^{2}dxdt=&\sum_{i=1}^{m+n+1}\int_{G}|\nabla_{G}(vU\omega_{i})|^{2}dxdt\nonumber \\
 =&\sum_{i=1}^{m+n+1}\int_{G}|U\omega_{i}\nabla_{G}v+v\nabla_{G}(U\omega_{i}) |^{2}dxdt\nonumber \\
 =&\sum_{i=1}^{m+n+1}\left(\int_{G}|\nabla_{G}v|^{2}U^{2}\omega_{i}^{2}dxdt+ \int_{G}|\nabla_{G}(U\omega_{i})|^{2}v^{2}dxdt+\right.\nonumber \\
 &\left. \frac{1}{2}\int_{G}\langle \nabla_{G}(U\omega_{i})^{2}, \nabla_{G}v^{2} \rangle dxdt\right)\nonumber \\
 =&\sum_{i=1}^{m+n+1}\left(\int_{G}|\nabla_{G}v|^{2}U^{2}\omega_{i}^{2}dxdt-\int_{G}v^{2}U\omega_{i}\Delta_{G}(U\omega_{i})dxdt\right)\nonumber \\
 =&\int_{G}|\nabla_{G}v|^{2}U^{2}dxdt+\frac{m(Q+2)}{4}\int_{G}v^{2}U^{\frac{2Q}{Q-2}}dxdt. \label{3.5}
\end{align}
To get the last equality, we use (\ref{2.8}).
On the other hand, by (\ref{b2.1}), we have
\begin{align}
\int_{G}|\nabla_{G}v|^{2}U^{2}dxdt=&\int_{G}\left|\nabla_{G}\frac{u}{U}\right|^{2}U^{2}dxdt \nonumber \\
=&\int_{G}|\nabla_{G}u|^{2}dxdt- \frac{m(Q-2)}{4}\int_{G}\frac{u^{2}}{(1+\frac{|x|^{2}}{4})^{2}+|t|^{2}}dxdt.  \label{3.6}
\end{align}
Substituting (\ref{3.6}) into (\ref{3.5}), we obtain
\begin{align*}
\sum_{i=1}^{m+n+1}\int_{G}|\nabla_{G}(u\omega_{i})|^{2}dxdt=\int_{G}|\nabla_{G}u|^{2}dxdt+m
\int_{G}\frac{u^{2}}{(1+\frac{|x|^{2}}{4})^{2}+|t|^{2}}dxdt.
\end{align*}
This completes the proof of Lemma \ref{lm3.4}.

\end{proof}

Now we can give the proof of Theorem \ref{th1.2}. The idea is due to Frank and Lieb \cite{fl1,fl2} and Hang and Wang \cite{hang}. \
\\

\textbf{Proof of Theorem \ref{th1.2}}. Let $2\leq p<\frac{2Q}{Q-2}$ and $u_{p}$ be a  positive solution of (\ref{3.1}).
The 2nd variation of the functional $\mathcal{F}_{p}$ around $u_{p}$ shows that
\begin{align*}
&\int_{G}|\nabla_{G}f|^{2}dxdt\int_{G}u_{p}^{p}U(x,t)^{\frac{2Q}{Q-2}-p}dxdt-\\
&
(p-1)
\int_{G}|\nabla_{G}u_{p}|^{2}dxdt \int_{G}u_{p}^{p-2}U_{\lambda,\eta}(x,t)^{\frac{2Q}{Q-2}-p}f^{2}dxdt\geq0
\end{align*}
for any $f$ with
\begin{align*}
\int_{G}u_{p}^{p}U_{\lambda,\eta}(x,t)^{\frac{2Q}{Q-2}-p}fdxdt=0.
\end{align*}
By Lemma \ref{lm3.3},  we may choose $f=u_{p}\omega_{i}$, $i=1,2,\cdots,m+n+1$.
Summing the corresponding inequalities for all such $f$'s yields,  in view of (\ref{2.7}) and Lemma \ref{lm3.4},
\begin{align*}
0\leq&\sum_{i=1}^{m+n+1}\int_{G}|\nabla_{G}(u_{p}\omega_{i})|^{2}dxdt-(p-1)\int_{G}|\nabla_{G}u_{p}|^{2}dxdt\\
=&m\int_{G}\frac{u_{p}^{2}}{(1+\frac{|x|^{2}}{4})^{2}+|t|^{2}}dxdt-(p-2)
\int_{G}|\nabla_{G}u_{p}|^{2}dxdt,
\end{align*}
i.e.
\begin{align*}
&(p-2)\left(\int_{G}|\nabla_{G}u_{p}|^{2}dxdt-\frac{m(Q-2)}{4}\int_{G}\frac{u_{p}^{2}}{(1+\frac{|x|^{2}}{4})^{2}+|t|^{2}}dxdt\right)\\
\leq&\frac{m(Q-2)}{4}\left(\frac{2Q}{Q-2}-p\right)\int_{G}\frac{u_{p}^{2}}{(1+\frac{|x|^{2}}{4})^{2}+|t|^{2}}dxdt\\
\leq&\frac{m(Q-2)}{4}\left(\frac{2Q}{Q-2}-p\right)\left(\int_{G}u_{p}^{p}U(x,t)^{\frac{2Q}{Q-2}-p}dxdt\right)^{\frac{2}{p}}
\left(\int_{G}U(x,t)^{\frac{2Q}{Q-2}}dxdt\right)^{1-\frac{2}{p}}\\
=&\frac{m(Q-2)}{4}\left(\frac{2Q}{Q-2}-p\right)
\left(\int_{G}U(x,t)^{\frac{2Q}{Q-2}}dxdt\right)^{1-\frac{2}{p}}\rightarrow0,\; p\nearrow \frac{2Q}{Q-2}.
\end{align*}
To get the last equality, we use the fact $\int_{G}u_{p}^{p}U(x,t)^{\frac{2Q}{Q-2}-p}dxdt=1$.
Therefore, by Lemma \ref{lm2.1}, we obtain
\begin{align*}
\int_{G}|\nabla_{G}u_{p}|^{2}dxdt-\frac{m(Q-2)}{4}\int_{G}\frac{u_{p}^{2}}{(1+\frac{|x|^{2}}{4})^{2}+|t|^{2}}dxdt\rightarrow0,\; p\nearrow \frac{2Q}{Q-2},
\end{align*}
or equivalently,
\begin{align*}
\int_{G}|\nabla_{G}(U^{-1}u_{p})|^{2}U^{2}dxdt\rightarrow0,\; p\nearrow \frac{2Q}{Q-2}.
\end{align*}
So
we can choose a sequence $\{p_{k}: k=1,2,\cdots \}$
such that $ p_{k}\nearrow \frac{2Q}{Q-2}$ and   $u_{p_{k}}$ converges to a nonzero function $c_{0}U$
(for reader's convenience, we prove it in Lemma \ref{lm3.5}). Thus $c_{0}U$  is an  extremal function of
 \begin{align*}
  \Lambda =\inf\left\{\int_{G}|\nabla_{G}u|^{2}dxdt : \; \int_{G}|u|^{\frac{2Q}{Q-2}}dxdt=1\right\}.
 \end{align*}
 The value $S_{m,n}$ has been calculated in \cite{gv1}, Theorem 1.6. The proof of Theorem \ref{th1.2} is thereby completed.

 \begin{lemma}\label{lm3.5}
Let $u_{p} (2\leq p<\frac{2Q}{Q-2})$ be a  positive solution of (\ref{3.1}). If
\begin{align*}
\int_{G}|\nabla_{G}u_{p}|^{2}dxdt-\frac{m(Q-2)}{4}\int_{G}\frac{u_{p}^{2}}{(1+\frac{|x|^{2}}{4})^{2}+|t|^{2}}dxdt\rightarrow0,\; p\nearrow \frac{2Q}{Q-2},
\end{align*}
then there exists $c_{0}>0$ and a sequence $\{p_{k}: k=1,2,\cdots \}$
such that $ p_{k}\nearrow \frac{2Q}{Q-2}$ and
\begin{align*}
\int_{G}|\nabla_{G}(u_{p_{k}}-c_{0}U)|^{2}dxdt\rightarrow0,\;k\rightarrow\infty.
\end{align*}
 \end{lemma}
\begin{proof}
By Lemma \ref{lm2.1}, $\mu_{1}=\frac{m(Q-2)}{4}$ is simple with eigenfunction
$U$ of (\ref{1.8}) with $\lambda=1$ and $\eta=0$.
Decompose $u_{p}$ as
\begin{align}\label{3.7}
u_{p}=\lambda_{p}U+v_{p}
\end{align}
with
\begin{align*}
\lambda_{p}=\frac{\int_{G}U^{\frac{Q+2}{Q-2}}u_{p}dxdt}{\int_{G}U^{\frac{2Q}{Q-2}}dxdt}>0.
\end{align*}
Then $v_{p}\perp U$, i.e.
\begin{align}\label{3.8}
\int_{G}U^{\frac{4}{Q-2}}\cdot Uv_{p}dxdt=\int_{G}U^{\frac{Q+2}{Q-2}}v_{p}dxdt=0,\; \int_{G}\langle\nabla_{G} U, \nabla_{G}v_{p}\rangle dxdt=0.
\end{align}
Therefore, we have
\begin{align}\label{3.9}
\int_{G}|\nabla_{G}v_{p}|^{2}dxdt\geq \mu_{2}\int_{G}U^{\frac{4}{Q-2}}\cdot v^{2}_{p}dxdt= \mu_{2}\int_{G}\frac{v_{p}^{2}}{(1+\frac{|x|^{2}}{4})^{2}+|t|^{2}}dxdt,
\end{align}
where $\mu_{2}$ is the second eigenvalue
of (\ref{1.8}) with with $\lambda=1$ and $\eta=0$. We compute, by using (\ref{3.8}) and (\ref{3.9}),
\begin{align*}
&\int_{G}|\nabla_{G}u_{p}|^{2}dxdt-\frac{m(Q-2)}{4}\int_{G}\frac{u_{p}^{2}}{(1+\frac{|x|^{2}}{4})^{2}+|t|^{2}}dxdt\\
=&\int_{G}\left(\lambda_{p}^{2}|\nabla_{G}U|^{2}+|\nabla_{G}v_{p}|^{2}\right)dxdt-\mu_{1}\int_{G}\frac{\lambda_{p}^{2}U^{2}+
v_{p}^{2}}{(1+\frac{|x|^{2}}{4})^{2}+|t|^{2}}dxdt\\
=&\int_{G}|\nabla_{G}v_{p}|^{2}dxdt-\mu_{1}\int_{G}\frac{v_{p}^{2}}{(1+\frac{|x|^{2}}{4})^{2}+|t|^{2}}dxdt\\
=&\frac{\mu_{1}}{\mu_{2}}\left(\int_{G}|\nabla_{G}v_{p}|^{2}dxdt- \mu_{2}\int_{G}\frac{u_{p}^{2}}{(1+\frac{|x|^{2}}{4})^{2}+|t|^{2}}dxdt\right)+\\
&\frac{\mu_{2}-\mu_{1}}{\mu_{2}}\int_{G}|\nabla_{G}v_{p}|^{2}dxdt\\
\geq&\frac{\mu_{2}-\mu_{1}}{\mu_{2}}\int_{G}|\nabla_{G}v_{p}|^{2}dxdt.
\end{align*}
Therefore,
\begin{align}
&\int_{G}|\nabla_{G}v_{p}|^{2}dxdt\nonumber\\
\leq&\frac{\mu_{2}}{\mu_{2}-\mu_{1}}
\left[\int_{G}|\nabla_{G}u_{p}|^{2}dxdt-\frac{m(Q-2)}{4}\int_{G}\frac{u_{p}^{2}}{(1+\frac{|x|^{2}}{4})^{2}+|t|^{2}}dxdt\right]\nonumber\\
&\rightarrow 0,\; p\nearrow \frac{2Q}{Q-2}. \label{3.10}
\end{align}

On the other hand, by (\ref{3.7}),  Minkowski's inequalities,  H\"older's inequality  and (\ref{1.1}), we have
\begin{equation}\label{3.11}
  \begin{split}
 &\lambda_{p}\left(\int_{G}U(x,t)^{\frac{2Q}{Q-2}}dxdt\right)^{\frac{1}{p}}\\
=&
\left(\int_{G}(u_{p}-v_{p})^{p}U(x,t)^{\frac{2Q}{Q-2}-p}dxdt\right)^{\frac{1}{p}}\\
\leq&\left(\int_{G}u_{p}^{p}U(x,t)^{\frac{2Q}{Q-2}-p}dxdt\right)^{\frac{1}{p}}+
\left(\int_{G}|v_{p}|^{p}U(x,t)^{\frac{2Q}{Q-2}-p}dxdt\right)^{\frac{1}{p}}\\
=&1+\left(\int_{G}|v_{p}|^{p}U(x,t)^{\frac{2Q}{Q-2}-p}dxdt\right)^{\frac{1}{p}}\\
\leq&1+\left(\int_{G}|v_{p}|^{\frac{2Q}{Q-2}}dxdt\right)^{\frac{Q-2}{2Q}}
\left(\int_{G}U^{\frac{2Q}{Q-2}}dxdt\right)^{\frac{1}{p}-\frac{Q-2}{2Q}}\\
\leq&1+C\left(\int_{G}|\nabla_{G}v_{p}|^{2}dxdt\right)^{\frac{1}{2}}
\left(\int_{G}U^{\frac{2Q}{Q-2}}dxdt\right)^{\frac{1}{p}-\frac{Q-2}{2Q}}.
  \end{split}
\end{equation}
Substituting (\ref{3.10}) into (\ref{3.11}), we obtain
\begin{align*}
\limsup_{p\nearrow \frac{2Q}{Q-2}}\lambda_{p}\leq \left(\int_{G}U(x,t)^{\frac{2Q}{Q-2}}dxdt\right)^{-\frac{Q-2}{2Q}}.
\end{align*}
Therefore, there exists $c_{0}\geq 0$ and a sequence $\{p_{k}: k=1,2,\cdots \}$
such that $ p_{k}\nearrow \frac{2Q}{Q-2}$ and
\begin{align}\label{3.12}
\lambda_{p_{k}}\rightarrow c_{0},\; k\rightarrow\infty.
\end{align}

We claim that
\begin{align}\label{3.13}
\int_{G}|\nabla_{G}(u_{p_{k}}-c_{0}U)|^{2}dxdt\rightarrow0,\;k\rightarrow\infty.
\end{align}
In fact, by using (\ref{3.7}),  (\ref{3.10}) and (\ref{3.12}), we obtain
\begin{align*}
&\int_{G}|\nabla_{G}(u_{p_{k}}-c_{0}U)|^{2}dxdt\\
=&\int_{G}|\nabla_{G}(v_{p_{k}}+(\lambda_{p_{k}}-c_{0})U)|^{2}dxdt\\
=&\int_{G}|\nabla_{G}v_{p_{k}}|^{2}dxdt+(\lambda_{p_{k}}-c_{0})^{2}\int_{G}|\nabla_{G}U|^{2}dxdt\rightarrow0,\;k\rightarrow\infty.
\end{align*}
This proves the claim.

Finally, we show that $c_{0}>0$. In fact, if $c_{0}=0$, then by (\ref{3.13}),
\begin{align*}
\int_{G}|\nabla_{G}u_{p_{k}}|^{2}dxdt\rightarrow0,\;k\rightarrow\infty.
\end{align*}
On the other hand, by H\"older's inequality and (\ref{1.1}), we obtain
\begin{align*}
1=&\left(\int_{G}u_{p_{k}}^{p_{k}}U(x,t)^{\frac{2Q}{Q-2}-p_{k}}dxdt\right)^{\frac{1}{p_{k}}}\\
\leq&\left(\int_{G}|u_{p_{k}}|^{\frac{2Q}{Q-2}}dxdt\right)^{\frac{Q-2}{2Q}}
\left(\int_{G}U^{\frac{2Q}{Q-2}}dxdt\right)^{\frac{1}{p_{k}}-\frac{Q-2}{2Q}}\\
\leq&C\left(\int_{G}|\nabla_{G}u_{p_{k}}|^{2}dxdt\right)^{\frac{1}{2}}
\left(\int_{G}U^{\frac{2Q}{Q-2}}dxdt\right)^{\frac{1}{p_{k}}-\frac{Q-2}{2Q}}\rightarrow0,\;k\rightarrow\infty,
\end{align*}
which is a contradiction. So $c_{0}>0$.
The proof of Lemma \ref{lm3.5} is thereby completed.
\end{proof}

Finally, we give the proof of Theorem \ref{th1.3}.\
\\

\textbf{Proof of Theorem \ref{th1.3}}  A simple scaling argument shows that
the eigenvalues do not depend on $\lambda$ and $\eta$. So we may
assume  $\lambda=1$ and $\eta=0$.

From Lemma \ref{lm2.1} we know that $\mu_{1}=\frac{m(Q-2)}{4}$ is simple with eigenfunction
$U$.

Nextly, we  show $\mu_{2}\geq \frac{m(Q+2)}{4}$.  Let $V\neq 0$ be a eigenfunction of
$\mu_{2}$. Then
\begin{equation}\label{3.14}
\mu_{2}=\frac{\int_{G}|\nabla_{G} V|^{2}dxdt}{\int_{G} U^{\frac{4}{Q-2}}V^{2}dxdt}.
\end{equation}
Furthermore, since $V\perp U$, we  have
\begin{equation}\label{3.15}
\int_{G}\langle\nabla_{G} U, \nabla_{G} V\rangle
dxdt=0,\;\;\int_{G}
U^{\frac{4}{Q-2}}\cdot UVdxdt=\int_{G} U^{\frac{Q+2}{Q-2}}Vdxdt=0.
\end{equation}
Set
\[
\Phi(\epsilon)=\frac{\int_{G}|\nabla_{G} (U+\epsilon
V)|^{2}dxdt}{\left(\int_{G} |U+\epsilon
V|^{\frac{2Q}{Q-2}}dxdt\right)^{\frac{Q-2}{Q}}},\;\;\epsilon\in\mathbb{R}.
\]
By Theorem \ref{th1.2},  $U$ is an extremal function of
Folland-Stein inequality (\ref{1.4}). So we have  $\Phi'(0)=0$ and $\Phi''(0)\geq0$.  We compute
\begin{align*}
\Phi'(\epsilon)
=&2\frac{\int_{G}\langle\nabla_{G}
(U+\epsilon V), \nabla_{G} V\rangle
dxdt}{\left(\int_{G} |U+\epsilon
V|^{\frac{2Q}{Q-2}}dxdt\right)^{\frac{Q-2}{Q}}}-\\
&2\frac{\int_{G}|\nabla_{G}
(U+\epsilon V)|^{2}dxdt}{\left(\int_{G} |U+\epsilon
V|^{\frac{2Q}{Q-2}}dxdt\right)^{\frac{2Q-2}{Q}}}\int_{G} |U+\epsilon
V|^{\frac{4}{Q-2}}(U+\epsilon V)Vdxdt\\
=& \Phi_{1}(\epsilon) -\Phi_{2}(\epsilon),
\end{align*}
where
\begin{align*}
\Phi_{1}(\epsilon) =&2\frac{\int_{G}\langle\nabla_{G}
(U+\epsilon V), \nabla_{G} V\rangle
dxdt}{\left(\int_{G} |U+\epsilon
V|^{\frac{2Q}{Q-2}}dxdt\right)^{\frac{Q-2}{Q}}};\\
\Phi_{2}(\epsilon)=&2\frac{\int_{G}|\nabla_{G}
(U+\epsilon V)|^{2}dxdt}{\left(\int_{G} |U+\epsilon
V|^{\frac{2Q}{Q-2}}dxdt\right)^{\frac{2Q-2}{Q}}}\int_{G} |U+\epsilon
V|^{\frac{4}{Q-2}}(U+\epsilon V)Vdxdt.
\end{align*}
By using  (\ref{3.15}), we have
\begin{align*}
\Phi'_{1}(0) =&2\frac{\int_{G}|\nabla_{G} V|^{2}
dxdt}{\left(\int_{G}
|U|^{\frac{2Q}{Q-2}}dxdt\right)^{\frac{Q-2}{Q}}}-4\frac{\int_{G}\langle\nabla_{G}
U, \nabla_{G} V\rangle dxdt}{\left(\int_{G}
|U|^{\frac{2Q}{Q-2}}dxdt\right)^{\frac{2Q-2}{Q}}}\int_{G} U^{\frac{Q+2}{Q-2}}Vdxdt\\
=&2\frac{\int_{G}|\nabla_{G} V|^{2}
dxdt}{\left(\int_{G}
|U|^{\frac{2Q}{Q-2}}dxdt\right)^{\frac{Q-2}{Q}}};
\\
\Phi_{2}'(0) =&4\frac{\int_{G}\langle\nabla_{G} U, \nabla_{G}
V\rangle dxdt}{\left(\int_{G}
|U|^{\frac{2Q}{Q-2}}dxdt\right)^{\frac{2Q-2}{Q}}}\int_{G} U^{\frac{Q+2}{Q-2}}Vdxdt
\\
&-\frac{8(Q-1)}{Q-2}\frac{\int_{G}|\nabla_{G}
U|^{2}dxdt}{\left(\int_{G}
|U|^{\frac{2Q}{Q-2}}dxdt\right)^{\frac{3Q-2}{Q}}}\left(\int_{G}
U^{\frac{4}{Q-2}}V^{2}dxdt\right)^{2}\\
&+\frac{2(Q+2)}{Q-2}\frac{\int_{G}|\nabla_{G}
U|^{2}dxdt}{\left(\int_{G} U^{\frac{2Q}{Q-2}}dxdt\right)^{\frac{2Q-2}{Q}}}\int_{G} U^{\frac{4}{Q-2}}V^{2}dxdt\\
=&\frac{2(Q+2)}{Q-2}\frac{\int_{G}|\nabla_{G}
U|^{2}dxdt}{\left(\int_{G} U^{\frac{2Q}{Q-2}}dxdt\right)^{\frac{2Q-2}{Q}}}\int_{G} U^{\frac{4}{Q-2}}V^{2}dxdt.
\end{align*}
Therefore,
\begin{align*}
0\leq\Phi''(0)=&\Phi_{1}'(0)-\Phi_{2}'(0)\\
=&2\frac{\int_{G}|\nabla_{G} V|^{2}
dxdt}{\left(\int_{G}
|U|^{\frac{2Q}{Q-2}}dxdt\right)^{\frac{Q-2}{Q}}}-
\frac{2(Q+2)}{Q-2}\frac{\int_{G}|\nabla_{G}
U|^{2}dxdt}{\left(\int_{G} U^{\frac{2Q}{Q-2}}dxdt\right)^{\frac{2Q-2}{Q}}}\int_{G} U^{\frac{4}{Q-2}}V^{2}dxdt,
\end{align*}
i.e.
\begin{align}\label{3.16}
\frac{\int_{G}|\nabla_{G}
V|^{2}dxdt}{\int_{G} |U|^{\frac{4}{Q-2}}V^{2}dxdt}\geq \frac{Q+2}{Q-2}\frac{\int_{G}|\nabla_{G}
U|^{2}dxdt}{\int_{G} |U|^{\frac{2Q}{Q-2}}dxdt}.
\end{align}
Combing (\ref{3.14}) and (\ref{3.16}) yields
\[
\mu_{2}\geq \frac{Q+2}{Q-2}\frac{\int_{G}|\nabla_{G}
U|^{2}dxdt}{\int_{G} |U|^{\frac{2Q}{Q-2}}dxdt}=\frac{Q+2}{Q-2}\mu_{1}=\frac{m(Q+2)}{4}.
\]
On the other hand, by (\ref{2.5}), $\{\partial_{\lambda}U_{\lambda,\eta}|_{\lambda=1,\eta=0},\;\nabla_{\eta}U_{\lambda,\eta}|_{\lambda=1,\eta=0}\}$ are eigenfunctions of $\frac{m(Q+2)}{4}$.
So $\mu_{2}=\frac{m(Q+2)}{4}$. This completes the proof of Theorem \ref{th1.3}.

\section*{Acknowledgements}
The author wishes  to thank the referees for their very helpful comments and suggestions which improved the exposition of the paper.

\end{document}